\documentclass[a4paper,12pt, reqno]{amsart}
\usepackage{amsmath, amsfonts, amssymb, amsthm}
\usepackage[all]{xy}
\usepackage{txfonts}
\usepackage{hyperref}
\usepackage{textcomp}

\theoremstyle{plain}
 \newtheorem{theorem}{Theorem}[section]
 
 \newtheorem{corollary}[theorem]{Corollary}
 \newtheorem{lemma}[theorem]{Lemma}
 \newtheorem{proposition}[theorem]{Proposition}

\newtheorem{refinedquestion}[theorem]{Refined Question}
\newcommand{\GL}{{\rm GL}}
\newcommand{\SL}{{\rm SL}}
\newcommand{\F}{\mathbb{F}}
\newcommand{\C}{\mathbb{C}}

\newcommand{\Z}{\mathbb{Z}}
\title{A question on splitting of metaplectic covers}
\author{Shiv Prakash Patel}
\date{\today}
\address{School of Mathematics\\ Tata Institute of Fundamental
Research\\ Homi Bhabha Road, Colaba\\Mumbai 400005, India}

\email{shiv@math.tifr.res.in}


\begin{document}
\maketitle
\begin{abstract}
Let $E/F$ be a quadratic extension of a non-Archimedian local field. Splitting of the 2-fold metaplectic cover of ${\rm Sp}_{2n}(F)$ when restricted to various subgroups of ${\rm Sp}_{2n}(F)$ plays an important role in application of the Weil representation of the metaplectic group. In this paper  we prove the splitting of the metaplectic cover of ${\rm GL}_{2}(E)$ over the subgroups ${\rm GL}_{2}(F)$ and $D_{F}^{\times}$, where $D_{F}$ is the quaternion division algebra with center $F$, as a first step in our study of the restriction of representations of metaplectic cover of ${\rm GL}_{2}(E)$ to ${\rm GL}_{2}(F)$ and $D_{F}^{\times}$. These results were suggested to the author by Professor Dipendra Prasad.
\end{abstract}
\section{Introduction}
 This paper will be concerned with certain 2-fold covers of ${\rm GL}_{2}(E)$, where $E$ is a non-Archimedian local field to be called the metaplectic covering of ${\rm GL}_{2}(E)$. We recall that there is a unique (up to isomorphism) non-trivial 2-fold cover of ${\rm SL}_{2}(E)$ called the metaplectic cover and denoted by $\widetilde{{\rm SL}_{2}(E)}$ in this paper, but there are many inequivalent 2-fold coverings of ${\rm GL}_{2}(E)$ which extend this 2-fold covering of ${\rm SL}_{2}(E)$.  We fix a covering of ${\rm GL}_{2}(E)$ as follows. Observe that ${\rm GL}_{2}(E)$ is the semi-direct product of ${\rm SL}_{2}(E)$ and $E^{\times}$, where $E^{\times}$ sits inside ${\rm GL}_{2}(E)$ as $e \mapsto \left( \begin{matrix} e & 0 \\ 0 & 1 \end{matrix} \right)$. This action of $E^{\times}$ on ${\rm SL}_{2}(E)$ lifts uniquely to an action of $E^{\times}$ on  $\widetilde{{\rm SL}_{2}(E)}$. Denote $\widetilde{{\rm GL}_{2}(E)} = \widetilde{{\rm SL}_{2}(E)}  \rtimes E^{\times}$ and call this the metaplectic cover of ${\rm GL}_{2}(E)$. Thus the metaplectic cover of ${\rm GL}_{2}(E)$ that we consider in this paper is that cover of ${\rm GL}_{2}(E)$ which extends the metaplectic cover of ${\rm SL}_{2}(E)$ and is further split on the subgroup $\left\{ \left( \begin{matrix} e & 0 \\ 0 & 1 \end{matrix} \right) : e \in E^{\times} \right\}$.\\
Given a central extension of a group $G$ by $\Z/2\Z$ say 
\[
0 \longrightarrow \Z/2\Z \longrightarrow G' \longrightarrow G \longrightarrow 1
\]
there is a natural central extension, say $G''$, of $G$ by $\C^{\times}$, given by
\[
G'' := G' \times_{\Z/2\Z} \C^{\times} :=\dfrac{G' \times \C^{\times}}{<(-1,-1)>}
\]
which sits in the following exact sequence
\begin{displaymath}
\xymatrix{ \{1\} \ar[r] & \Z/2\Z \ar[r] \ar@{^{(}->}[d] & G' \ar[r] \ar@{^{(}->}[d]  & G \ar[r] \ar@{=}[d] & \{1\} \\ 
 \{1\} \ar[r] & \C^{\times} \ar[r] & G'' \ar[r] & G \ar[r] & \{1\} }  
\end{displaymath}
This $\C^{\times}$-central extension of $G$ is said to be obtained from the 2-fold cover of $G$. It is well known that $\C^{\times}$-covers tend to be easier to analyse and which is what we shall do in this paper. \\
Let $F$ be a non-Archimedian local field of characteristic zero and $E$ is a quadratic extension of $F$. Let $D_{F}$ denote the unique quaternion division algebra with center $F$. Note that $D_{F}^{\times} \hookrightarrow \GL_{2}(E)$ given by fixing an isomorphism $D_{F} \otimes E \cong M_{2}(E)$. By Skolem-Noether theorem, such an embedding is uniquely determined upto conjugation by elements of $\GL_{2}(E)$. The main theorem of this paper is the following:
\begin{theorem} \label{main theorem}
Let $E$ be a quadratic extension of a non-Archimedian local field and $\widetilde{{\rm GL}_{2}(E)}$ be the two-fold metaplectic covering of ${\rm GL}_{2}(E)$.  Then:
\begin{enumerate}
\item The two-fold metaplectic covering splits over the subgroup ${\rm GL}_{2}(F)$.
\item The $\C^{\times}$-covering obtained from $\widetilde{{\rm GL}_{2}(E)}$ splits over the subgroup $D_{F}^{\times}$. 
\end{enumerate} 
\end{theorem}
Note that a quadratic extension $L$ of $F$ gives rise to two embeddings of $L$ in $M(2,E)$ as in the diagram below:
\begin{displaymath}
\xymatrix{ & D_{F} \ar@{^{(}->}[rd] & \\
L \ar@{^{(}->}[rd] \ar@{^{(}->}[ru] & & M(2,E). \\
& M(2,F) \ar@{^{(}->}[ru] & }
\end{displaymath}
 By Skolem-Noether theorem, any two embeddings of $L \otimes E$ in $M(2,E)$ and hence of $L$ are conjugate in $M(2,E)$ by ${\rm GL}_{2}(E)$. \\
 
The  following refined question was formulated by Dipendra Prasad.
\begin{refinedquestion}
Does there exist a natural identification of the set of splittings of the $\C^{\times}$-cover of ${\rm GL}_{2}(E)$ restricted to ${\rm GL}_{2}(F)$ and set of splittings restricted to $D_{F}^{\times}$ (in either of the two cases the set of splittings is a principal homogeneous space over the character group of $F^{\times}$) such that for any quadratic extension $L$ of $F$, the two embeddings of $L^{\times}$ in the $\C^{\times}$-covers $\widetilde{{\rm GL}_{2}(E)}$ (which we take to be $\C^{\times}$-central extension of ${\rm GL}_{2}(E)$) 
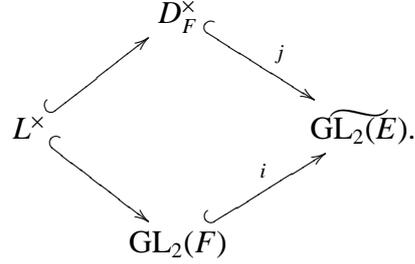
\begin{displaymath}
\xymatrix{ & D_{F}^{\times} \ar@{^{(}->}[rd]^{j} & \\
L^{\times} \ar@{^{(}->}[rd] \ar@{^{(}->}[ru] & & \widetilde{{\rm GL}_{2}(E)}. \\
& {\rm GL}_{2}(F) \ar@{^{(}->}[ru]^{i} & }
\end{displaymath}
are conjugate in the $\C^{\times}$ cover $\widetilde{\GL_{2}(E)}$ ?
\end{refinedquestion}
We are not able to handle the refined question, and will only contend with the proof of the splitting of the  metaplectic cover of ${\rm GL}_{2}(E)$ restricted to $D_{F}^{\times}$. However the above refined question plays an important role in harmonic analysis relating the pair $(\widetilde{{\rm GL}_{2}(E)}, {\rm GL}_{2}(F))$ with the pair $(\widetilde{{\rm GL}_{2}(E)}, D_{F}^{\times})$. \\ \\
We briefly say a few words about the proofs. The proof for ${\rm GL}_{2}(F)$ is straight forward from the explicit knowledge of the cocycle defining the metaplectic cover. For any quadratic extension $L$ of $F$, we know that the embedding $L^{\times} \hookrightarrow D_{F}^{\times}$ is conjugate inside ${\rm GL}_{2}(E)$ to the embedding of $L^{\times}$ inside ${\rm GL}_{2}(E)$ realized as $L^{\times} \hookrightarrow {\rm GL}_{2}(F) \hookrightarrow {\rm GL}_{2}(E)$. Since the metaplectic cover of ${\rm GL}_{2}(E)$ is split over ${\rm GL}_{2}(F)$, it is split in particular over $L^{\times}$ for any quadratic extension $L$ of $F$. Thus we know that the restriction of the metaplectic cover of ${\rm GL}_{2}(E)$ to $D_{F}^{\times}$ has the property that it splits over $L^{\times}$ of any quadratic extension $L$ of $F$. This is the key property to be used in the proofs below. 
The coverings considered in this paper are topological coverings which are locally split, i.e. all the coverings considered here split when restricted to a small neighbourhood  of identity. Although it appears that we deal with abstract coverings we do not emphasize every time on the fact that these are topological coverings.

\section{Splitting over ${\rm GL}_{2}(F)$} \label{section:GL2}
We prove the following proposition:
\begin{proposition} \label{prop:A}
Let $E$ be a quadratic extension of a non-Archimedian local field $F$. Then the metaplectic 2-fold cover $\widetilde{{\rm GL}_{2}(E)}$ of ${\rm GL}_{2}(E)$, as described in the introduction, splits over the subgroup ${\rm GL}_{2}(F)$.
\end{proposition}
\begin{proof}
 To prove that the covering $\widetilde{{\rm GL}_{2}(E)}$ of ${\rm GL}_{2}(E)$ splits over ${\rm GL}_{2}(F)$, we observe that the two-cocycle $\beta$ which defines the two-fold metaplectic cover satisfies $\beta(\sigma, \tau) = 1$ for all $\sigma, \tau \in {\rm GL}_{2}(F)$, i.e., the cocycle is identically 1 when restricted to ${\rm GL}_{2}(F)$. One knows that the defining expression of the cocycle $\beta$ involves only quadratic Hilbert symbols of the field $E$. The proof will follow if we prove that restriction of the quadratic Hilbert symbol of $E$ restricted to $F$ is identically 1, which is the content of the next lemma.
\end{proof}
\begin{lemma} \label{lemma:B}
If we denote the quadratic Hilbert symbol of the field $E$ by $( \cdot, \cdot)_{E}$. Then
\[
 (a,b)_{E} = 1 \mbox{ for all } a,b \in F^{\times}.
\]
\end{lemma}
\begin{proof}
Let $(\cdot, \cdot)_{F}$ denotes the quadratic Hilbert symbol of the field $(\cdot, \cdot)_{F}$. Then it is well known that for $a \in F^{\times}$ and $b \in E^{\times}$, we have
\[
(a,b)_{E} = (a, Nb)_{F}.
\]
Hence for $a, b \in F^{\times}$ we have
\[
  (a,b)_{E} = (a, Nb)_{F} = (a, b^{2})_{F} = 1. \qedhere
\]
\end{proof}
\section{Splitting over ${\rm SL}_{1}(D_{F})$} \label{section:SL1D}
Recall that $D_{F}$ denotes the unique quaternion division algebra over the field $F$ and ${\rm SL}_{1}(D_{F})$ is the subgroup of norm 1 elements. Fix an embedding $E \hookrightarrow D_{F}$ through which $D_{F}$ can be realized as a two dimensional vector space over $E$ with $E$ acting on $D_{F}$ on the left and $D_{F}$ acting on itself on the right. This gives rise to an embedding $D_{F}^{\times} \hookrightarrow {\rm GL}_{2}(E)$. Since  ${\rm SL}_{1}(D_{F})$ is compact, we can assume that ${\rm SL}_{1}(D_{F}) \subset {\rm GL}_{2}(\mathcal{O}_{E})$. It is well known that if the residue characteristic of $F$ is odd, then the two-fold metaplectic cover $\widetilde{{\rm GL}_2(E)}$ of ${\rm GL}_2(E)$ splits over ${\rm GL}_{2}(\mathcal{O}_{E})$ and hence over ${\rm SL}_{1}(D_{F})$. Such a simple minded proof does not work for $p=2$, however, we prove in this section that the $\C^{\times}$ metaplectic cover of ${\rm SL}_{2}(E)$ splits when restricted to ${\rm SL}_{1}(D_{F})$.

\begin{proposition} \label{prop:C}
 Restriction of the non-trivial 2-fold cover of ${\rm SL}_{4}(F)$ to ${\rm SL}_{2}(E)$ remains non-trivial, hence gives the unique non-trivial 2-fold cover of ${\rm SL}_{2}(E)$.
\end{proposition}
\begin{proof}
The proposition amounts to the assertion that there is a commutative diagram between the unique 2-fold covers of ${\rm SL}_{2}(E)$ and ${\rm SL}_{4}(F)$:
\begin{displaymath}
\xymatrix{ 0 \ar[r] & \Z/2\Z \ar[r] \ar@{=}[d] & \widetilde{{\rm SL}_{2}(E)} \ar[r] \ar@{^{(}->}[d] & {\rm SL}_{2}(E) \ar[r] \ar@{^{(}->}[d] & 1 \\
                 0 \ar[r] & \Z/2\Z \ar[r]  & \widetilde{{\rm SL}_{4}(F)} \ar[r] & {\rm SL}_{4}(F) \ar[r]  & 1 }
\end{displaymath}
This follows from the generality that the transfer map from $K_{2}(E)/2K_{2}(E)$ to $K_{2}(F)/2K_{2}(F)$ is an isomorphism \cite{Mil71}; we omit the details.
\end{proof}
\begin{corollary} \label{corollary:D}
 The $\C^{\times}$-cover of  ${\rm SL}_{2}(E)$ obtained from $\widetilde{{\rm SL}_{2}(E)}$ splits over ${\rm SL}_{1}(D_{F})$. \qedhere
\end{corollary}
\begin{proof}
From the proposition \ref{prop:C}, restriction of the 2-fold cover from ${\rm SL}_{4}(F)$ to ${\rm SL}_{2}(E)$ remains non-trivial. Since we have inclusion of groups
\begin{displaymath}
\xymatrix{ {\rm SL}_{2}(E) \ar@{^{(}->}[r] & {\rm Sp}_{4}(F) \ar@{^{(}->}[r] & {\rm SL}_{4}(F), }
\end{displaymath}
and all these groups have a unique non-trivial 2-fold cover, we deduce that the unique non-trivial 2-fold cover of ${\rm SL}_{4}(F)$ restricts to give the unique non-trivial 2-fold cover of ${\rm Sp}_{4}(F)$ which in turn restricts to the unique non-trivial 2-fold cover of ${\rm SL}_{2}(E)$. Now we use the inclusion of the groups
\begin{displaymath}
\xymatrix{ {\rm SL}_{1}(D_{F}) \ar@{^{(}->}[r] & {\rm SL}_{2}(E) \ar@{^{(}->}[r] & {\rm Sp}_{4}(F) }
\end{displaymath}
and use a result of Kudla \cite[~Theorem 3.1]{Kud94} according to which the restriction of the $\C^{\times}$-covering of ${\rm Sp}_{4}(F)$ to ${\rm U}(2)$ splits. (The result of Kudla is valid for any unitary group ${\rm U}(n)$ defined by a skew hermitian form in $n$ variables over $E$ which comes with a natural embedding in ${\rm Sp}_{2n}(F)$). If we take the hermitian form in  2 variables which is anisotropic, then for the corresponding unitary group ${\rm U}(2), {\rm SU}(2) \cong {\rm SL}_{1}(D_{F})$.  As a result, restriction of the $\C^{\times}$-covering from ${\rm SL}_{2}(E)$ to ${\rm SL}_{1}(D)=SU(2) \subset U(2)$ splits.	\qedhere
\end{proof}

\section{Splitting over $D_{F}^{\times}$} \label{section:D*}
In this section we prove the splitting of $\C^{\times}$-cover of ${\rm GL}_{2}(E)$ obtained from $\widetilde{{\rm GL}_{2}(E)}$ when restricted to $D_{F}^{\times}$. 
\subsection{Even residue characteristic case} 
Note the following short exact sequence
\[
 1 \longrightarrow {\rm SL}_{1}(D_{F}) \longrightarrow D_{F}^{\times} \longrightarrow F^{\times} \longrightarrow 1. \label{exact:A} \tag{A}
\]
Let $\C^{\times}$ be the trivial $D_{F}^{\times}$-module. Then $H^{2}(D_{F}^{\times}, \C^{\times})$ classifies central extensions of $D_{F}^{\times}$ by the group $\C^{\times}$.  The Hochschild-Serre spectral sequence arising from (\ref{exact:A}) gives a filtration on $H^{2}(D_{F}^{\times}, \C^{\times}):$
\[
 H^{2}(D_{F}^{\times}, \C^{\times}) = F^{0} \supseteq F^{1} \supseteq F^{2} \supseteq 0
\]
with $F^{0}/F^{1} = E_{\infty}^{0,2}$, $F^{1}/F^{2}= E_{\infty}^{1,1}$ and $F^{2} = E_{\infty}^{2,0}$, where
\begin{center} 
 $E_{2}^{0,2} = H^{0}(F^{\times}, H^{2}({\rm SL}_{1}(D_{F}), \C^{\times})), $ \\
 $E_{2}^{1,1} = H^{1}(F^{\times}, H^{1}({\rm SL}_{1}(D_{F}), \C^{\times})), $ \\
 $E_{2}^{2,0} = H^{2}(F^{\times}, H^{0}({\rm SL}_{1}(D_{F}), \C^{\times})) $.
\end{center}
Consider the embedding $D_{F}^{\times} \hookrightarrow {\rm GL}_{2}(E)$ and denote the restriction of the central extension of ${\rm GL}_{2}(E)$ to $D_{F}^{\times}$ as well as the corresponding element of $H^{2}(D_{F}^{\times}, \C^{\times})$ by $\beta$. In the section \ref{section:SL1D} we proved that $\beta$ restricted to ${\rm SL}_{1}(D_{F})$ is trivial, therefore $\beta \in F^{1}$. In even residue characteristic, since we are dealing with a cohomology class of order 2 (or 1), the following result of C. Riehm \cite{Riehm70} implies that $\beta$ must be trivial in $F^{1}/F^{2}$.
\begin{proposition}[Riehm] \label{prop:Riehm}
Let $G_{0} = {\rm SL}_{1}(D_{F})$ and $G_{i}$ for $i \geq 1$ be the standard congruence subgroup of $G_{0}$. The
\[
[G_{0}, G_{0}] = G_{1}.
\]
In particular, the character group of ${\rm SL}_{1}(D_{F})$ is a finite cyclic group of order prime to $p$.
\end{proposition}
Thus in the even residue characteristic an element of $H^{2}(D_{F}^{\times}, \C^{\times})$ of order 2 which is trivial when restricted to ${\rm SL}_{1}(D_{F})$ arises from the inflation of an element of $H^{2}(F^{\times}, \C^{\times})$.
An element of $H^{2}(F^{\times}, \C^{\times})$ is represented by a central extension 
\[
1 \rightarrow \C^{\times} \rightarrow \widetilde{F^{\times}} \rightarrow F^{\times} \rightarrow 1.
\]
The proof of the splitting of the $\C^{\times}$-metaplectic cover of  ${\rm GL}_{2}(E)$ restricted to $D_{F}^{\times}$ will be completed in even residue characteristic once we prove the following lemma.
\begin{lemma} \label{lemma:F}
A $\C^{\times}$-covering of $D_{F}^{\times}$ coming from a $\C^{\times}$ covering of $F^{\times}$ via the norm map, which is trivial on $L^{\times}$ for all quadratic extensions $L$ of $F$, is trivial.
\end{lemma}
\begin{proof}
Suppose that there exists a {\it non-trivial} $\C^{\times}$-covering of $D_{F}^{\times}$ coming froma $\C^{\times}$-central extension $\widetilde{F^{\times}}$ of $F^{\times}$ via the norm map, which is trivial on $L^{\times}$ for all quadratic extension $L$ of $F$. Thus there are two elements $e_1, e_2 \in \widetilde{F^{\times}}$ which do not commute. Look at the images, say, $f_1, f_2$ of $e_1, e_2$ in $F^{\times}$. Let $\bar{f_1}, \bar{f_2}$ be images of $f_1, f_2$ in $F^{\times}/F^{\times 2}$. Since residue characteristic of $F$ is even, $F^{\times}/F^{\times 2}$ is a vector space over $\Z/2\Z$ of dimension $\geq 3$. Therefore given any two elements $\bar{f_1}, \bar{f_2} \in F^{\times}/F^{\times 2}$, there exist a subgroup $F_{1} \hookrightarrow F^{\times}$ of index 2 containing $f_1, f_2$. By local class field theory, there exists a unique quadratic extension $M$ of $F$ with ${\rm Norm}_{M/F}(M^{\times}) = F_{1}$. Now we use the fact given to us that the central extension of $D_{F}^{\times}$ that we are considering is trivial on $L^{\times}$ for any quadratic extension $L$ of $F$, in particular on $M^{\times}$. Hence the inverse image of $M^{\times}$ in the central extension must be abelian, a contradiction to the construction of $M$.
\end{proof} 
\subsection{Odd residue characteristic case}
In this subsection, we assume that the residue characteristic $p$ of $F$ is odd. We first introduce more notation. Let $\mathcal{O}_{D_{F}}$ be the maximal compact subring of $D_{F}$ and $\mathcal{P}_{D_{F}}$ be the maximal ideal of $\mathcal{O}_{D_{F}}$. Let $D_{F}^{\times}(1) := 1 + \mathcal{P}_{D_{F}}$. Note that $D_{F}^{\times}(1)$ is a normal pro-$p$ subgroup in $D_{F}^{\times}$. Since $p$ is odd and $D_{F}^{\times}(1)$ is a normal pro-$p$ subgroup
\[
H^{2}(D_{F}^{\times}, \Z/2\Z) \cong H^{2}(D_{F}^{\times}/D_{F}^{\times}(1), \Z/2\Z).
\]
In other words, every 2-fold central extension of $D_{F}^{\times}$ arises as a pull back of a 2-fold central extension $D_{F}^{\times}/D_{F}^{\times}(1)$. The structure of the group $D_{F}^{\times}/D_{F}^{\times}(1)$ is $\mathbb{F}_{q^2}^{\times} \rtimes \Z$, where $\mathbb{F}_{q^2}$ is the finite field with $q^2$ elements and $\Z$ operates on $\F_{q^2}^{\times}$ by powers of the Frobenius map $x \mapsto x^q$. This group sits in the following short exact sequence 
\[
0 \rightarrow \mathbb{F}_{q^2}^{\times} \rightarrow G':=D_{F}^{\times}/D_{F}^{\times}(1) \rightarrow \Z \rightarrow 0.
\]
Using this description of the group we prove the following proposition.
\begin{proposition} \label{prop:H}
\begin{enumerate}
\item[(A)] We have
\[
H^{2}(D_{F}^{\times}, \Z/2\Z)  \cong \Z/2\Z \oplus \Z/2\Z.
\]
\item[(B)] If we denote 2-torsion element of $H^{2}(D_{F}^{\times}, \C^{\times})$ by $H^{2}(D_{F}^{\times}, \C^{\times})[2]$ then
\[
H^{2}(D_{F}^{\times}, \C^{\times})[2] = \Z/2\Z.
\]
\end{enumerate}
\end{proposition}
\begin{proof}
\begin{enumerate} 
\item[(A)]  Since $G' = \F_{q^2}^{\times} \rtimes \Z$ and $\Z$ has cohomological dimension 1, the Hochschild-Serre spectral sequence $E_{2}^{i,j} = H^{i}(\Z, H^{j}(\F_{q^2}^{\times}, \Z/2\Z))$ calculating the cohomology of $G'$  has $E_{2}^{1,1}=E_{\infty}^{1,1}$, $E_{2}^{0,2} = E_{\infty}^{0,2}$ and $E_{2}^{2,0} = E_{\infty}^{2,0} =0$. Therefore
\[
0 \longrightarrow H^{1}(\Z, H^{1}(\F_{q^2}^{\times}, \Z/2\Z)) \longrightarrow H^{2}(G', \Z/2\Z) \longrightarrow H^{2}(\F_{q^2}^{\times}, \Z/2\Z)^{\Z} \longrightarrow 0.
\]
Since $H^{1}(\F_{q^2}^{\times}, \Z/2\Z) = \Z/2\Z$ and $H^{2}(\F_{q^2}^{\times}, \Z/2\Z) = \Z/2\Z$ and since $\Z$ must act trivially on $\Z/2\Z$, we get
\[
0 \rightarrow H^{1}(\Z, \Z/2\Z) \rightarrow H^{2}(G', \Z/2\Z) \rightarrow \Z/2\Z \rightarrow 0.
\]
which proves part (A) of the proposition.
\item[(B)] This is evident from the short exact sequence of next Lemma \ref{lemma:I}.
\end{enumerate}
This proves the proposition.
\end{proof}
By proposition \ref{prop:H} there are four non-isomorphic two-fold coverings of the group $D_{F}^{\times}$. The lemma below proves that one of these non-trivial 2-fold covers becomes trivial as a $\C^{\times}$-cover. 
\begin{lemma} \label{lemma:I}
We have the following short exact sequence 
\[
0 \longrightarrow \dfrac{H^{1}(D_{F}^{\times}, \C^{\times})}{2H^{1}(D_{F}^{\times}, \C^{\times})} \longrightarrow H^{2}(D_{F}^{\times}, \Z/2\Z) \longrightarrow H^{2}(D_{F}^{\times}, \C^{\times})[2] \longrightarrow 0
\]
with
\[
\dfrac{H^{1}(D_{F}^{\times}, \C^{\times})}{2H^{1}(D_{F}^{\times}, \C^{\times})} \cong \Z/2\Z
\]
where for any abelian group $A$, $A[2]=\{ a \in A : 2a=0 \}$. 
\end{lemma}
\begin{proof}
The short exact sequence can be deduced from the long exact sequence of cohomology groups of $D_{F}^{\times}$ arising from the following short exact sequence 
\begin{displaymath}
\xymatrix{ 0  \ar[r] & \Z/2\Z \ar[r] & \C^{\times} \ar[r]^{2}  & \C^{\times} \ar[r] & 0 }.
\end{displaymath}
Since $[D_{F}^{\times}, D_{F}^{\times}]= \SL_{1}(D_{F})$ and $D_{F}^{\times}/\SL_{1}(D_{F}) \cong F^{\times}$, the second statement follows from the fact that the character group of $D_{F}^{\times}$, i.e. $H^{1}(D_{F}^{\times}, \C^{\times})$, is the same as the character group of $F^{\times}$, and in the odd residue characteristic, it is easy to see that 
\[
\dfrac{H^{1}(F^{\times}, \C^{\times})}{2H^{1}(F^{\times}, \C^{\times})} \cong  \Z/2\Z. \qedhere
\] 
\end{proof}
\begin{proposition} \label{prop:J}
Let $M$ be the quadratic unramified extension of $F$ with $M  \hookrightarrow D_{F}$. Then a 2-fold cover of $D_{F}^{\times}$ which remains non-trivial with  $\C^{\times}$ coefficients does not split over the subgroup $M^{\times} \hookrightarrow D_{F}^{\times}$.
\end{proposition}
\begin{proof}
Let  $M^{\times}(1)=1+\mathcal{P}_{M}$. As $M$ is a quadratic unramified extension of $F$, we have
\[
M^{\times}/M^{\times}(1) \cong \F_{q^2}^{\times} \times \Z.
\]
Since $\Z$ has cohomological dimension 1, by Kunneth theorem
\[
\begin{array}{lll}
 H^{2}(M^{\times}/M^{\times}(1), \Z/2\Z) &\cong& H^{2}(\F_{q^2}^{\times}, \Z/2\Z) \oplus \left( H^{1}(\F_{q^2}^{\times}, \Z/2\Z) \otimes H^{1}(\Z, \Z/2\Z) \right) \\
 &=& \Z/2\Z \oplus \Z/2\Z.
\end{array}
\]
Since $M^{\times}/M^{\times}(1) \cong \F_{q^2}^{\times} \times \Z$, its character group is isomorphic to $\widehat{\F_{q^2}^{\times}} \times \C^{\times}$. So once again as in Lemma \ref{lemma:I}, we get the following short exact sequence:
\[
0 \rightarrow \Z/2\Z= \dfrac{H^{1}(M^{\times}/M^{\times}(1), \C^{\times})}{2 H^{1}(M^{\times}/M^{\times}(1), \C^{\times})} \rightarrow H^{2}(M^{\times}, \Z/2\Z) \rightarrow H^{2}(M^{\times}, \C^{\times}) \rightarrow 0
\]
By considering the embedding $M^{\times}/M^{\times}(1) \hookrightarrow D_{F}^{\times}/D_{F}^{\times}(1)=G'$, we get the following exact sequences with connecting homomorphisms
\begin{displaymath}
 \xymatrix{ 0 \ar[r] & \dfrac{H^{1}(G', \C^{\times})}{2 H^{1}(G', \C^{\times})} \ar[r] \ar[d]^{f} & H^{2}(G', \Z/2\Z) \ar[r] \ar[d]^{g}  & H^{2}(G', \C^{\times})[2] \ar[r] \ar[d]^{h} & 0 \\
                  0 \ar[r] & \dfrac{H^{1} \left( \frac{M^{\times}}{M^{\times}(1)}, \C^{\times} \right) }{2 H^{1} \left( \frac{M^{\times}}{M^{\times}(1)}, \C^{\times}\right)} \ar[r] & H^{2} \left( \frac{M^{\times}}{M^{\times}(1)}, \Z/2\Z \right) \ar[r] & H^{2} \left( \frac{M^{\times}}{M^{\times}(1)}, \C^{\times} \right) [2] \ar[r] &  0 } \label{Diagram:**} \tag{**}
\end{displaymath}
In the next lemma we prove that $h$ is injective. This proves the proposition.
\end{proof}
\begin{lemma} \label{lemma:L}
The right most vertical map $h : H^{2}(G', \C^{\times})[2] \rightarrow H^{2}(M^{\times}, \C^{\times})[2]$ in the above diagram \ref{Diagram:**} is an isomorphism.
\end{lemma}
 \begin{proof}
Consider the short exact sequence which appeared in the proof of Proposition \ref{prop:H} with $\Z/2\Z$ replaced by $\C^{\times}$,
 \[
0 \longrightarrow H^{1}(\Z, H^{1}(\F_{q^2}^{\times}, \C^{\times})) \longrightarrow H^{2}(G', \C^{\times}) \longrightarrow H^{2}(\F_{q^2}^{\times}, \C^{\times})^{\Z} \longrightarrow 0.
\]
This combined with the fact that second cohomology of a cyclic group with $\C^{\times}$-coefficients is zero, implies that
\[
H^{2}(G', \C^{\times}) = H^{1}(\Z, H^{1} (\F_{q^2}^{\times}, \C^{\times})) = H^{1}(\Z, \widehat{\F_{q^2}^{\times}}).
\]
Similarly
\[
H^{2}(M^{\times}, \C^{\times}) = H^{1}(2\Z, H^{1} (\F_{q^2}^{\times}, \C^{\times})) = H^{1}(2\Z, \widehat{\F_{q^2}^{\times}}).
\] 
We need to prove that the restriction map 
\[
H^{1}(\Z, \widehat{\F_{q^2}^{\times}})[2] \cong \Z/2\Z \longrightarrow H^{1}(2\Z, \widehat{\F_{q^2}^{\times}})[2] \cong \Z/2\Z
\]
is injective. For this,
consider the following short exact sequence
 \begin{displaymath}
\xymatrix{ 0 \ar[r] & \Z \ar[r]^{2}  &\Z \ar[r] & \Z/2\Z \ar[r] & 0 } 
 \end{displaymath}
The above exact sequence gives rise to the following inflation-restriction exact sequence
\[
0 \longrightarrow H^{1}(\Z/2\Z, \widehat{\F_{q^2}^{\times}}) \longrightarrow H^{1}(\Z, \widehat{\F_{q^2}^{\times}}) \longrightarrow H^{1}(2\Z, \widehat{\F_{q^2}^{\times}}).
\] 
By the next lemma there is an isomorphism of $\widehat{\F_{q^2}^{\times}}$ with $\F_{q^2}^{\times}$ preserving the natural $Gal(\F_{q^2}/\F_{q})$ action on these groups.
Hence by Hilbert's theorem 90 we get that $H^{1}(\Z/2\Z, \widehat{\F_{q^2}^{\times}} )= 0$. So the map 
\[
H^{1}(\Z, \widehat{\F_{q^2}^{\times}}) \rightarrow H^{1}(2\Z, \widehat{\F_{q^2}^{\times}})
\]
is injective and hence in particular on 2-torsions 
\[
H^{1}(\Z, \widehat{\F_{q^2}^{\times}})[2] \rightarrow H^{1}(2\Z, \widehat{\F_{q^2}^{\times}})[2].
\]
This proves that the map $h$ is non-zero and an isomorphism.
\end{proof}
\begin{lemma} \label{lemma:M}
There is an isomorphism of $\widehat{\F_{q^d}^{\times}}$ with $\F_{q^d}^{\times}$ such that the natural Galois action of $Gal(\F_{q^d}/\F_{q})$ on $\widehat{\F_{q^d}^{\times}}$ becomes the inverse of the natural action of $Gal(\F_{q^d}/\F_{q})$ on $\F_{q^d}^{\times}$ (where by ``inverse" of an action of an abelian group $G$ on a module $M$, we mean $g * m = (g^{-1})m$).
\end{lemma}
\begin{proof}
Since the $Gal(\F_{q^d}/\F_{q})$ operates by $x \mapsto x^{q}$ on $\F_{q^d}^{\times}$, the proof of the lemma is clear.
\end{proof}

\section*{Acknowledgements}
The author would like to thank Professor Dipendra Prasad for his help in writing the paper. He also thanks Professor Sandeep Varma for many useful conversations. The author is supported by Tata Institute of Fundamental Research, Mumbai by the Ph. D. scholarship.

%\begin{acknowledgements} \normalfont
%The author would like to thank Professor Dipendra Prasad for his help in writing the paper. He also thanks Professor Sandeep Varma for many useful conversations.
%\end{acknowledgements}

\end{document}